\documentclass[12pt,reqno]{amsart}

\usepackage{amssymb}

\newtheorem{theorem}{Theorem}[section]
\newtheorem{proposition}[theorem]{Proposition}
\newtheorem{lemma}[theorem]{Lemma}

\theoremstyle{remark}
\newtheorem{remark}[theorem]{Remark}

\numberwithin{equation}{section}

\begin{document}
\baselineskip=15.5pt

\title[Flat vector bundles over parallelizable manifolds and stability]{On the stability of flat 
complex vector bundles over parallelizable manifolds}

\author[I. Biswas]{Indranil Biswas}

\address{School of Mathematics, Tata Institute of Fundamental
Research, Homi Bhabha Road, Mumbai 400005, India}

\email{indranil@math.tifr.res.in}

\author[S. Dumitrescu]{Sorin Dumitrescu}

\address{Universit\'e C\^ote d'Azur, CNRS, LJAD, France}

\email{dumitres@unice.fr}

\author[M. Lehn]{Manfred Lehn}

\address{Institut f\"ur Mathematik, Johannes Gutenberg Universit\"at Mainz, 55099 Mainz, Germany}

\email{lehn@mathematik.uni-mainz.de}



\date{}

\begin{abstract}

We investigate the flat holomorphic vector bundles over compact complex parallelizable manifolds $G 
/ \Gamma$, where $G$ is a complex connected Lie group and $\Gamma$ is a cocompact lattice in it. 
The main result proved here is a structure theorem for flat holomorphic vector bundles $E_\rho$ 
associated to any irreducible representation $\rho \,: \,\Gamma \,\longrightarrow\, \text{GL}(r, 
{\mathbb C})$. More precisely, we prove that $E_{\rho}$ is holomorphically isomorphic to a vector 
bundle of the form $E^{\oplus n}$, where $E$ is a stable vector bundle. All the rational Chern 
classes of $E$ vanish, in particular, its degree is zero.

We deduce a stability result for flat holomorphic vector bundles $E_{\rho}$ of rank 2 over $G/ \Gamma$. If an irreducible representation
$\rho \,:\, \Gamma
\,\longrightarrow\, \text{GL}(2, \mathbb {C})$ satisfies the condition
that the induced homomorphism $\Gamma\, \longrightarrow\, {\rm 
PGL}(2, {\mathbb C})$ does not extend to a homomorphism from
$G$, then $E_{\rho}$ is proved to be stable.
\end{abstract}
\maketitle

\vspace{0.1cm}

\noindent{{ R\'esum\'e.}} Nous \'etudions les fibr\'es holomorphes plats sur les vari\'et\'es parall\'elisables compactes $G / \Gamma$ (avec $G$ groupe de Lie connexe complexe et $\Gamma$ r\'eseau cocompact). Notre r\'esultat principal d\'ecrit les fibr\'es holomorphes plats $E_{\rho}$ associ\'es \`a des repr\'esentations irr\'eductibles $\rho \,:\, \Gamma \,\longrightarrow\, \text{GL}(r, {\mathbb C})$. Nous d\'emontrons que ces fibr\'es $E_{\rho}$ sont isomorphes \`a une somme
directe $E^{\oplus n}$, avec $E$ fibr\'e vectoriel stable de degr\'e zero. 

Nous en d\'eduisons un r\'esultat de stabilit\'e concernant les fibr\'es holomorphes plats $E_{\rho}$ de rang $2$ sur les quotients $G /\Gamma$. Si $\rho \,: \,\Gamma \,\longrightarrow\, \text{GL}(2, {\mathbb C})$ est 
une repr\'esentation irr\'eductible telle que le
morphisme induit $\rho'\,:\, \Gamma\, \longrightarrow\, {\rm PGL}(2, {\mathbb C})$ ne s'\'etend pas \`a $G $, alors $E_{\rho}$ est stable.

{\bf Version fran\c caise abr\'eg\'ee}

Nous \'etudions les fibr\'es plats holomorphes sur les vari\'et\'es parall\'elisables compactes $G 
/ \Gamma$, avec $G$ groupe de Lie complexe connexe et $\Gamma$ r\'eseau cocompact. Ces fibr\'es 
plats de rang $r$ sont donn\'es par des repr\'esentations $\rho \,:\, \Gamma
\,\longrightarrow\, \text{GL}(r, 
{\mathbb C})$. Un tel fibr\'e est holomorphiquement trivial si et seulement si
le morphisme $\rho$ s'\'etend en un morphisme 
de groupes de Lie $G \,\longrightarrow\, \text{GL}(r, {\mathbb C})$ (voir, par exemple, \cite[p.~801, Proposition 3.1]{Wi}).

Nous nous int\'eressons \`a la notion de stabilit\'e de ces fibr\'es. M\^eme si pour $G$ non 
ab\'elien, les quotients $G / \Gamma$ ne sont pas k\"ahleriens, ces vari\'et\'es portent des m\' 
etriques {\it balanc\'ees} (i.e. qui ont la propri\'et\'e $d\omega^{m-1}=0$, avec $m$ la 
dimension complexe de la vari\'et\'e) \cite{AB, Bi}. Par rapport \`a ces m\'etriques les notions 
classiques de degr\'e, pente (=degr\'e divis\'e par le rang) et (semi-)stabilit\'e et 
polystabilit\'e au sens des pentes se d\'efinissent comme dans le cas classique (projective ou k\"ahlerien) \cite{HL}.

Le r\'esultat principal de cette note d\'etermine la structure des fibr\'es holomorphes plats (de 
rang $r$) $E_{\rho}$ au-dessus de $G / \Gamma$ qui sont construits \`a partir de repr\'esentations 
irr\'eductibles $\rho \,:\, \Gamma \,\longrightarrow\, \text{GL}(r, {\mathbb C})$. Nous d\'emontrons que ces fibr\'es 
sont n\'ecesairement isomorphes \`a un fibr\'e de la forme $E^{\oplus n},$ avec $E$ fibr\'e 
vectoriel stable de degr\'e zero.

Dans la preuve nous utilisons un r\'esultat de \cite{Bi} qui donne la semistabilit\'e de 
$E_{\rho}$. Ensuite nous d\'emontrons la polystabilit\'e qui est une \'etape importante de la 
preuve.

Une cons\'equence du th\'eor\`eme pr\'ec\'edent est un r\'esultat de stabilit\'e pour les fibr\'es 
plats holomorphes de rang $2$ au-dessus de $G / \Gamma$. Plus pr\'ecis\'ement, si $\rho \,:\, \Gamma \,\longrightarrow\, 
\text{GL}(2, {\mathbb C})$ est une repr\'esentation irr\'eductible telle que la repr\'esentation induite 
$\rho'\,:\, \Gamma\, \longrightarrow\, {\rm PGL}(2, {\mathbb C})$ ne s'\'etend 
pas \`a $G$, alors $E_{\rho}$ est stable.

\section{Introduction}\label{sec1}

An interesting class of compact non-K\"ahler manifolds which generalizes compact complex tori consists
of those manifolds whose holomorphic tangent bundle is holomorphically trivial. By a result 
of Wang~\cite{Wa}, those so-called {\it parallelizable} manifolds are known to be biholomorphic to 
a quotient of a complex connected Lie group $G$ by a cocompact lattice $\Gamma$ in $G$. Those 
quotients $G/ \Gamma$ are K\"ahler exactly when $G$ is abelian (and, consequently, the quotient is 
a complex torus).

In particular, for $G$ nonabelian, the above quotients are not algebraic. Moreover, for $G$ 
semi-simple, the corresponding parallelizable manifolds are known to have algebraic dimension 
zero (meaning that the only meromorphic functions on $G / \Gamma$ are the constant ones).

For those parallelizable manifolds of algebraic dimension zero, Ghys's arguments in \cite{Gh2} 
prove that all foliations and all holomorphic distributions on $G / \Gamma$ are {\it 
homogeneous} (i.e., they all descend from $G$-right invariant foliations (respectively, distributions) 
on $G$). In particular, any complex subbundle of the holomorphic tangent bundle is isomorphic 
to a trivial vector bundle. It was also proved in \cite{DM} that all holomorphic geometric 
structures in Gromov's sense \cite{Gro} (constructed from higher order frame bundles) on 
parallelizable manifolds $G / \Gamma$ of algebraic dimension zero are also necessarily 
homogeneous (e.g. their pull-back on $G$ are $G$-right invariant).

The simplest examples of parallelizable manifolds of algebraic dimension zero are compact 
quotients of $\text{SL}(2, {\mathbb C})$. Those manifolds are closely related to the 
$3$-hyperbolic manifolds in a natural way. Indeed, $\text{PSL}(2, {\mathbb C})$ being the group 
of direct isometries of the hyperbolic $3$-space, the direct orthonormal frame bundle of a 
compact oriented hyperbolic $3$-manifold $V$ is diffeomorphic to a quotient $\text{PSL}(2, {\mathbb C}) 
/ \Gamma$ (with the lattice $\Gamma$ being isomorphic to the fundamental group of $V$). Those 
quotients are both geometrically interesting and very abundant. In particular, they can have 
arbitrarily large first Betti number (see \cite{Mi}, \cite[Section 6.2]{Gh}, or \cite{La}), 
meaning that the rank of the abelianization of $\Gamma$ (modulo torsion) can be arbitrarily 
large. This remark was used by Ghys \cite{Gh} in order to prove that the quotients of 
$\text{SL}(2, {\mathbb C})$ with first Betti number $\geq 1$ are not rigid (as complex 
manifolds). Ghys constructed the corresponding deformation space using group homomorphisms of 
$\Gamma$ into $\text{SL}(2, {\mathbb C}) \times \text{SL}(2, {\mathbb C})$ which are close to 
the natural embedding $\Gamma \,\subset\, \text{SL}(2, {\mathbb C})\times \{I_2\}
\, \subset\, \text{SL}(2, {\mathbb C}) \times \text{SL}(2, {\mathbb C})$, where
$I_2$ is the identity element. The images of those 
homomorphisms are (up to conjugacy) graphs $\gamma\,\longmapsto\, (\gamma,\, \rho(\gamma))$,
$\gamma \,\in\, \Gamma$, of homomorphisms $ \rho \,:\, \Gamma \,\longrightarrow\,
\text{SL}(2, {\mathbb C})$ that are close to the map of $\Gamma$ to the identity
element (the trivial homomorphism).

As soon as the first Betti number of $\Gamma$ is positive, there exists nontrivial group 
homomorphisms $u \,:\, \Gamma \,\longrightarrow\, {\mathbb C}$ and one can consider
$\rho\, :\, \Gamma \,\longrightarrow\, \text{SL}(2, {\mathbb C})$ defined by
$\gamma\, \longmapsto\, \exp (u(\gamma)\xi)$, 
with $\exp$ being the exponential map and $\xi$ some fixed element of $\text{sl}(2, {\mathbb C})$.
Notice that those homomorphisms do not extend to homomorphisms from
$\text{SL}(2, {\mathbb C})$; indeed, $\text{SL}(2, {\mathbb C})$ is a {\it perfect} group
(meaning generated by its commutators)
and hence $\text{SL}(2, {\mathbb C})$ does not admit nontrivial homomorphisms into abelian groups.
This implies that the 
associated flat holomorphic line bundle $E_{\rho}$ over $\text{SL}(2, {\mathbb C}) / \Gamma$ is 
nontrivial \cite[p.~801, Proposition 3.1]{Wi}.

For any $r\, \geq\, 2$, any any element $\xi$ of $\text{Lie}(\text{SL}(r, {\mathbb C}))$,
the previous homomorphisms $u \,:\, \Gamma \,\longrightarrow\, {\mathbb C}$ produce
group homomorphisms $\Gamma \,\longrightarrow\, \text{SL}(r, {\mathbb C})$,
$\gamma\, \longmapsto\, \exp(u(\gamma)\xi)$, taking values in a one parameter subgroup.

For cocompact lattices $\Gamma$ with first Betti number $\geq 2$, Ghys constructed in 
\cite{Gh} group homomorphisms from $\Gamma$ to $\text{SL}(2, {\mathbb C})$ which are close to 
the identity and such that the image is Zariski dense. One can also see \cite{La} in which the 
author constructs many cocompact lattices $\Gamma$ in $\text{SL}(2, {\mathbb C})$ admitting a 
surjective homomorphism onto a nonabelian free group (this also
implies that the lattices can have arbitrarily large first Betti number).
Since nonabelian free groups admit many linear irreducible representations, we get many linear 
irreducible representations of those $\Gamma$ furnishing nontrivial flat holomorphic vector 
bundles over the corresponding quotients $\text{SL}(2, {\mathbb C} )/ \Gamma$.

In this note we deal with flat holomorphic vector bundles over parallelizable manifolds $G / 
\Gamma$. Our main result (see Theorem \ref{thm1} in Section \ref{s2}) is a structure theorem 
for flat holomorphic vector bundles $E_\rho$ given by irreducible representations $\rho \,: \, 
\Gamma \,\longrightarrow\, \text{GL}(r, {\mathbb C})$. We prove that $E_{\rho}$ is isomorphic 
to a direct sum $E^{\oplus n}$, where $E$ is a stable vector bundle.
All the rational Chern classes of $E$ vanish, in particular, the degree of $E$ is zero (see
Remark \ref{rem2}).

As a consequence, we deduce in Section \ref{s3} a stability result for flat holomorphic vector 
bundles $E_{\rho}$ of rank two over $G / \Gamma$. If an
irreducible homomorphism $\rho \,:\, 
\Gamma \,\longrightarrow\, \text{GL}(2, \mathbb {C})$ has the property that
the homomorphism
$\Gamma\, \longrightarrow\, {\rm PGL}(2, {\mathbb C})$ obtained by composing $\rho$ 
with the natural projection of $\text{GL}(2,{\mathbb C})$ to ${\rm PGL}(2, {\mathbb C})$ does not 
extend to a homomorphism from $G$, then $E_{\rho}$ is stable.

It should be mentioned that the stability of vector bundles in the previous results is slope-stability 
with respect to a {\it balanced} metric on the parallelizable manifold (see 
definition and explanations in Section \ref{s2}).

In a further work the authors will address the question of (semi)stability of flat holomorphic 
vector bundles over Ghys's deformations of parallelizable manifolds $\text{SL}(2, {\mathbb C}) / 
\Gamma$ constructed in \cite{Gh}. It should be mentioned that generic small deformations of 
$\text{SL}(2, {\mathbb C}) / \Gamma$ do not admit nontrivial holomorphic $2$-forms (see Lemma 
3.3 in \cite{BD}) and, consequently, Corollary 8 in \cite{FY} implies that those generic small 
deformations admit balanced metrics.

\section{Irreducible representations and polystable bundles}\label{s2}

Let $G$ be a complex connected Lie group and $\Gamma\, \subset\, G$ a discrete subgroup
such that the quotient
\begin{equation}\label{h1}
M\, :=\, G/\Gamma
\end{equation}
is a compact (complex) manifold. Such a $\Gamma$ is called a cocompact lattice in $G$. The left--translation action of $G$ on itself produces a holomorphic action of
$G$ on the complex manifold $M$.
The Lie algebra of $G$ will be denoted by
$\mathfrak g$. Take a maximal compact subgroup $K\, \subset\, G$. Fix a
Hermitian form $h\,\in\, {\mathfrak g}^*\otimes\overline{\mathfrak g}^*$ on ${\mathfrak g}$
such that the action of $K$ on ${\mathfrak g}^*\otimes\overline{\mathfrak g}^*$, given
by the adjoint action of $K$ on ${\mathfrak g}$, fixes $h$. Consider the translations of
$h$ by the right--multiplication action of $G$ on itself. The resulting Hermitian
structure on $G$ descends to a Hermitian
structure on the quotient $M$ in \eqref{h1}. The $(1,\, 1)$--form on
$M$ corresponding to this Hermitian structure will be denoted by $\omega$. We know that
\begin{equation}\label{h2}
d\omega^{m-1}\,=\, 0\, ,
\end{equation}
where $m\, =\, \dim_{\mathbb C} M$ \cite[p.~277, Theorem~1.1(1)]{Bi}. Such a Hermitian metric
structure is called balanced (one can also see \cite{AB}).

For a torsion-free coherent analytic sheaf $F$ on $M$, define
$$
\text{degree}(F)\,:=\, \int_M c_1(\det F)\wedge \omega^{m-1}\, \in\, {\mathbb R}\, ,
$$
where $\det F$ is the determinant line bundle for $F$ \cite[Ch.~V, \S~6]{Ko}. From
\eqref{h1} it follows that the degree is well-defined. Indeed, any two
first Chern forms for $F$ differ by an exact $2$--form on $M$, and
$$
\int_M (d\alpha)\wedge \omega^{m-1}\,=\, -\int_M \alpha\wedge d\omega^{m-1}\,=\, 0\, .
$$
In fact, this show that degree is a topological invariant. Define $$\mu(F)\,:=\,
\frac{\text{degree}(F)}{\text{rank}(F)}\, \in\, {\mathbb R}\, $$ which is called the slope of $F$.
A torsion-free coherent analytic sheaf $F$ on $M$ is called \textit{stable} (respectively,
\textit{semistable}) if for every coherent analytic subsheaf $V\,\subset\, F$
such the $\text{rank}(V)\, \in\, [1\, ,\text{rank}(F)-1]$, the inequality
$\mu(V) \, <\, \mu(F)$ (respectively, $\mu(V) \, \leq\, \mu(F)$)
holds (see \cite[Ch.~V, \S~7]{Ko}). Hence, {\it throughout the paper, (semi)stability means slope-(semi)stability.}

A torsion-free coherent analytic sheaf $F$ is called
\textit{polystable} if the following two conditions hold:
\begin{enumerate}
\item $F$ is semistable, and

\item $F$ is a direct sum of stable sheaves.
\end{enumerate}

Consider any homomorphism
\begin{equation}\label{e2}
\rho\,:\, \Gamma\, \longrightarrow\, \text{GL}(r,{\mathbb C})\, .
\end{equation}
Let $(E_\rho\, ,\nabla^\rho)$ be the flat holomorphic vector bundle
or rank $r$ over $M$ associated to $\rho$. We recall that the total space of
$E_\rho$ is the quotient of $G
\times {\mathbb C}^r$ where two points $(z_1\, ,v_1)\, , (z_2\, ,v_2)\,\in\,
G\times {\mathbb C}^r$ are identified if there is an
element $\gamma\, \in\, \Gamma$ such that $z_2\,=\, z_1\gamma$ and $v_2\,=\,
\rho(\gamma^{-1})(v_1)$. We note that the fiber of $E_\rho$ over the point $e\Gamma\, \in\,
M$, where $e$ is the identity element of $G$, is identified with ${\mathbb C}^r$
by sending $w\, \in\, {\mathbb C}^r$ to the equivalence class of $(e,\, w)$.
The trivial connection on the trivial vector bundle
$G\times {\mathbb C}^r\,\longrightarrow\,
G$ descends to a connection on $E_\rho$ which is denoted by $\nabla^\rho$. The
left--translation action of $G$ on $G$ and the trivial
action of $G$ on ${\mathbb C}^r$ together define an action
of $G$ on $G
\times {\mathbb C}^r$. It descends to an action of $G$
on $E_\rho$. This action of $G$ on $E_\rho$
is a lift of the left--translation action of $G$ on $M$.
In particular, the holomorphic vector bundle $E_\rho$ is equivariant. 

A homomorphism $\rho$ as in \eqref{e2} is called \textit{reducible} if there is no
nonzero proper subspace of ${\mathbb C}^r$ preserved by $\rho(\Gamma)$ for the
standard action of $\text{GL}(r,{\mathbb C})$ on ${\mathbb C}^r$.

\begin{theorem}\label{thm1}
Assume that the homomorphism $\rho$ is irreducible. Then $E_\rho$ is isomorphic to
a vector bundle of the form $E^{\oplus n}$, where $E$ is a stable vector bundle.
\end{theorem}

\begin{proof}
Since $E_\rho$ admits a flat connections, namely $\nabla^\rho$, the rational Chern
class $c_1(E_\rho) \, \in\, H^2(M,\, {\mathbb Q})$ vanishes. Hence from the definition
of degree it follows that $\text{degree}(E_\rho)\, =\, 0$.
Since $E_\rho$ is equivariant, it follows that $E_\rho$ is
semistable \cite[p.~279, Lemma 3.2]{Bi}.

We will now prove that $E_\rho$ is polystable.

Let
\begin{equation}\label{f2}
W\, \subset\, E_\rho
\end{equation}
be the unique maximal polystable subsheaf of degree zero in $E_\rho$ given by
Proposition \ref{propsocle} in Section \ref{uniqueness}. Consider the action of $G$ on the 
equivariant bundle $E_\rho$. From the uniqueness of $W$
it follows that the action of $G$ on $E_\rho$ preserves the subsheaf $W$ 
in \eqref{f2}; indeed, the subsheaf $g\cdot W\, \subset\, E_\rho$ given by the action on $W$
of any fixed $g\, \in\, G$ is again a maximal polystable subsheaf of degree zero, and hence
$g\cdot W$ coincides with $W$.
This implies that $W$ is a subbundle of $E_\rho$ (the subset of $M$ over which 
$W$ is subbundle is preserved by the action of $G$ on $M$, and hence it must be entire $M$ as 
the action is transitive). As noted earlier, the fiber of $E_\rho$ over the point $e\Gamma\, 
\in\, M$ is identified with ${\mathbb C}^r$. The isotropy subgroup for $e\Gamma$ is $\Gamma$ 
itself. Since $W$ is preserved by the action of $G$ on $E_\rho$, we conclude that the subspace 
$W_{e\Gamma}\,\subset\, (E_\rho)_{e\Gamma}\,=\, {\mathbb C}^r$ is preserved by the action of 
$\Gamma$ on ${\mathbb C}^r$ given by $\rho$. Since $\rho$ is irreducible, it follows that 
$W_{e\Gamma}\,=\, (E_\rho)_{e\Gamma}$. This implies that $W\,=\, E_\rho$. Hence $E_\rho$ is 
polystable.

Let
$$
E_\rho\, =\, \bigoplus_{i=1}^n F_i
$$
be a decomposition of the polystable bundle $E_\rho$ into a direct sum of stable bundles.
Since $\text{degree}(E_\rho)\,=\, 0$, and $E_\rho$ is semistable, it follows that
$\text{degree}(F_i)\,=\, 0$ for all $1\, \leq\, i\, \leq\, n$.

Order the above vector
bundles $F_i$ in such a way that
\begin{enumerate}
\item $F_i$ is isomorphic to $F_1$ for all $1\, \leq\, i\, \leq\, \ell$, and

\item $F_i$ is not isomorphic to $F_1$ for every $\ell+1\, \leq\, i\, \leq\, n$.
\end{enumerate}
Note that $\ell$ may be $1$. Let
$$
F\, =\, \bigoplus_{i=1}^\ell F_i\, \subset\, \bigoplus_{i=1}^n F_i\,=\, E_\rho
$$
be the subbundle given by the direct sum of the first $\ell$ summands.
We will show that the flat connection $\nabla^\rho$ on $E_\rho$
preserves $F$.

Let $\Omega_M$ denote the holomorphic cotangent bundle of $M$. Consider the
composition
$$
F\, \hookrightarrow\, E_\rho \,\stackrel{\nabla^\rho}{\longrightarrow}\,
E_\rho\otimes \Omega_M\,\longrightarrow\, (E_\rho/F)\otimes \Omega_M\, ;
$$
it is an ${\mathcal O}_M$--linear homomorphism known as the second fundamental form
of $F$ for $\nabla^\rho$. We will denote this second fundamental form
by $S(\nabla^\rho,\, F)$. Now, $\Omega_M$ is trivial (a trivialization is given by
a right translation invariant trivialization of the holomorphic cotangent bundle of $G$).
Also, we have
$$
E_\rho/F \,=\, \bigoplus_{i=\ell+1}^n F_i\, .
$$
For any $\ell+1\, \leq\, i\, \leq\, n$, since $F_i$ is a stable bundle of degree
zero not isomorphic to the stable vector bundle $F_1$ of degree zero, it follows that
\begin{equation}\label{h00}
H^0(M, \, \text{Hom}(F_1,\, F_i))\,= H^0(M, \, F_i\otimes F^*_1)\,=\, 0\, .
\end{equation}
Since $F$ is a direct sum copies of $F_1$, and $\Omega_M$ is trivial, from \eqref{h00} we
conclude that $S(\nabla^\rho,\, F)\,=\, 0$. This implies that the connection $\nabla^\rho$ 
on $E_\rho$ preserves $F$.

Since $\nabla^\rho$ preserves $F$, and $\rho$ is irreducible, it follows that
$F\, =\, E_\rho$. This completes the proof of the theorem.
\end{proof}

\begin{remark}
If the homomorphism $\rho$ extends to a homomorphism $\widetilde{\rho}\, :\,
G\, \longrightarrow\, \text{GL}(r,{\mathbb C})$, then the vector bundle $E_\rho$ is
holomorphically trivial. Indeed, the map $G\times {\mathbb C}^r\, \longrightarrow\,
G\times {\mathbb C}^r$ that sends any $(z,\, v)$ to $(z,\, \widetilde{\rho}(z)(v))$
descends to a holomorphic isomorphism of $E_\rho$ with the trivial vector bundle $M\times
{\mathbb C}^r$ (recall that $E_\rho$ is a quotient of $G\times {\mathbb C}^r$). Therefore,
in that case the integer $n$ in Theorem \ref{thm1} is the rank $r$.
\end{remark}

\begin{remark}\label{rem2}
Since $E_\rho$ admits a flat connection, all the rational Chern classes of $E_\rho$ of
positive degree vanish \cite{At}. Therefore, the condition $E^{\oplus n}\,=\, E_\rho$ in Theorem
\ref{thm1} implies that all the rational Chern classes of $E$ of positive degree vanish.
\end{remark}

\section{Uniqueness of socle of semistable reflexive sheaves}\label{uniqueness}

The aim of this section is to prove the following Proposition \ref{propsocle} about semi-stable 
reflexive sheaves (see Definition 1.1.9 on page 6 in \cite{HL}). The analogous (weaker) statement 
for bundles is needed in the proof of Theorem \ref{thm1}.

\begin{proposition}\label{propsocle}
Consider a compact complex manifold $X$ equipped with a Gauduchon metric $\omega$.
Let $E$ be a semistable reflexive sheaf on $X$ of
slope $\mu$ (with respect to $\omega$). Then there is a unique polystable sheaf $E'\,\subset\, E$ with 
$\mu(E')\,=\,\mu$ that is maximal among all such subsheaves. 
\end{proposition}

\begin{proof}
Let $F\,\subset\, E$ be any subsheaf of slope $\mu(F)\,=\,\mu$.
Taking the double dual of the inclusion we see that $F^{**}$
embeds into $E^{**}=E$.
Moreover, since $F^{**}/F$ is a torsion sheaf it follows that 
$\mu(F)\,\leq\, \mu(F^{**})$. Indeed, by \cite[Ch.~V, pp.~166--167, Proposition~6.14]{Ko} the determinant
line bundle ${\mathcal L}$ of the torsion sheaf $F^{**}/F$ admits a nontrivial holomorphic section. Then we make use of the following Lemma (stated as
Proposition 1.3.5 on page 35 in \cite{LT}) which is a consequence of Poincar\'e-Lelong formula:

\begin{lemma} If the line bundle ${\mathcal L}$ admits a nontrivial holomorphic section $t$ with vanishing divisor $D_t$, then
${\rm degree}({\mathcal L}) \,=\, c \cdot {\rm Vol}_{\omega}(D_t)$, where $c$ is a positive constant and the volume ${\rm Vol}_{\omega}(D_t)$ of
$D_t$ is computed with the fixed Gauduchon metric $\omega$. In particular, if ${\mathcal L}$ is nontrivial and admits a nontrivial
holomorphic section, then ${\rm degree}({\mathcal L})\,>\,0.$
\end{lemma} 

Since $E$ is semistable,
it follows that $$\mu(F)\,=\,\mu(F^{**})\,=\,\mu\, ,$$ so that $F^{**}$ is semistable. Assume
that $F$ is stable and that $V \,\subset\, F^{**}$ is a non-trivial subsheaf of 
slope $\mu(V)\,=\,\mu$. Then $V \cap F$ is a subsheaf of the same slope and rank as $V$. 
Since $F$ is stable, one has $${\rm rank}(V)\,=\,{\rm rank}(F)
\,=\,{\rm rank}(F^{**})\, .$$ This shows that $F^{**}$
is stable as well. Therefore, every stable subsheaf of $E$ of slope
$\mu$ is contained in a reflexive stable subsheaf of the same rank and slope.

Let $E'\,\subset\, E$ be the sum of all stable reflexive subsheaves of slope $\mu$.
By the argument of the first paragraph, $E'$ contains all stable subsheaves of slope
$\mu$ in $E$. It suffices to show that $E'$ is polystable. Let $$E''\,:=\,E_1\oplus \cdots
\oplus E_s\,\subset\, E'$$ be a polystable subsheaf of $E'$ with a maximal number $s$ of stable
reflexive subsheaves $E_i$ of slope $\mu$. If $E''\,=\,E'$ there is nothing to show. 
Otherwise, there is a reflexive stable subsheaf $F\,\subset\, E'$ of slope $\mu$ 
not contained in $E''$. If $F\cap E''\,=\,0$, then $E''\oplus F\,\subset\, E'$, contradicting 
the maximality of $s$. Assume therefore that $F\cap E''\,\neq\, 0$. 

If $\mu(F\cap E'')\,<\,\mu$, then
$\mu(E''+F)\,=\,\mu(E''\oplus F/(E''\cap F))\,>\,\mu$, contradicting the semistability
of $E$. Hence $\mu(F\cap E'')\,=\,\mu$. Since $F$ is stable, $F\cap E''\,\subset\, F$
is a subsheaf of the same rank, so that $F/(F\cap E'')$ is a torsion module. Since
$E''$ is reflexive, the inclusion $F\cap E''\,\subset\, E''$ extends to $F\,\subset\, E''$,
contradicting the assumptions on $F$. 

This shows that $E'=E''$ is polystable. By construction, $E'$ is unique. 
\end{proof}

\section{Rank two flat bundles on quotients of $\text{SL}(2,{\mathbb C})$}\label{s3}

We set $r\,=\, 2$ in \eqref{e2}. Take any
irreducible $\rho$ as in Theorem \ref{thm1}. Let $\rho'\,:\, \Gamma\, \longrightarrow\,
{\rm PGL}(2, {\mathbb C})$ be the composition of $\rho$ with the canonical projection of
$\text{GL}(2,{\mathbb C})$ to ${\rm PGL}(2, {\mathbb C})$.

\begin{proposition}\label{prop1}
Assume that the homomorphism $\rho'$ does not extend to a homomorphism from
$G$. Then the vector bundle $E_\rho$ is stable.
\end{proposition}

\begin{proof}
Assume that $E_\rho$ is not stable. From Theorem \ref{thm1} it follows that
$E_\rho\,=\, L\oplus L$, where $L$ is a holomorphic line bundle. Hence the
projective bundle ${\mathbb P}(E_\rho)$ is holomorphically trivial. Now
\cite[p.~801, Proposition 3.1]{Wi} contradicts the given condition that
$\rho'$ does not extend to a homomorphism from $G$.
\end{proof}

\begin{remark}\label{rem-l}
In an earlier proof of Proposition \ref{prop1}, a longer argument was given after
proving that ${\mathbb P}(E_\rho)$ is holomorphically trivial. The referee pointed
out that \cite[p.~801, Proposition 3.1]{Wi} completes the proof at this point.
This also enabled us to remove the assumption $G\,=\, \text{SL}(2,{\mathbb C})$
in the earlier version.
\end{remark}

Let $\rho,\, \eta\, :\, \Gamma\, 
\longrightarrow\, \text{GL}(2,{\mathbb C})$ be two irreducible homomorphisms such that
\begin{itemize}
\item neither of the two corresponding homomorphisms $\rho',\, \eta'\, :\, \Gamma\, \longrightarrow
\, {\rm PGL}(2, {\mathbb C})$ extends to a homomorphism from $G$, and

\item for the action $\rho\otimes\eta^*$ of $\Gamma$ on $\text{End}({\mathbb C}^2)\,=\,
{\mathbb C}^2\otimes ({\mathbb C}^2)^*$,
where $\eta^*$ is the action of $\Gamma$ on $({\mathbb C}^2)^*$ corresponding
to its action on ${\mathbb C}^2$ given by $\eta$, the vector space
$\text{End}({\mathbb C}^2)$ decomposes into a direct sum of irreducible
$\Gamma$--modules such that all the direct summands are nontrivial and the
action of $\Gamma$ on none of them extends to an action of $G$.
\end{itemize}
{}From Proposition \ref{prop1} we know that the associated vector bundles
$E_\rho$ and $E_\eta$ are stable.

\begin{lemma}\label{lem1}
The two stable vector bundle $E_\rho$ and $E_\eta$ are not isomorphic.
\end{lemma}

\begin{proof}
It suffices to show that
\begin{equation}\label{s}
H^0(M,\, \text{Hom}(E_\eta,\, E_\rho))\,=\, 0\, .
\end{equation}
Since $\text{End}({\mathbb C}^2)$ decomposes into a direct sum of irreducible
$\Gamma$--modules, from Theorem \ref{thm1} we know that
\begin{equation}\label{fi}
\text{Hom}(E_\eta,\, E_\rho)\,=\, \bigoplus_{i=1}^b (F_i)^{\oplus d_i}\, ,
\end{equation}
where each $F_i$ is a stable vector bundle of degree zero (here $b$ is the
number of irreducible $\Gamma$--modules in $\text{End}({\mathbb C}^2)$). To prove
\eqref{s} it is enough to show that each $F_i$ is nontrivial, because a nontrivial
stable vector bundle of degree zero does not admit any nonzero section.

Let $V$ be an irreducible $\Gamma$ module and $E_V$ the corresponding vector bundle
on $M$. If $E_V$ is trivial, then
the action of $\Gamma$ on $V$ extends to an action of $G$ \cite[p.~801, Proposition 3.1]{Wi}. 
Therefore, each $F_i$ in \eqref{fi} is nontrivial.
\end{proof}

\section*{Acknowledgements}

We thank the referee for detailed comments, in particular on Proposition \ref{prop1}
(mentioned in Remark \ref{rem-l}). IB is supported by a J. C. Bose Fellowship.


\end{document}